\documentclass[12pt, reqno, twoside, a4paper]{amsart}
\usepackage{bardestani-freiberg-mertens-theorem}
\usepackage{bardestani-freiberg-mertens-theorem-notation}

\title[Mertens's theorem for splitting primes and more]
      {Mertens's theorem for splitting primes and more
       \\ \vspace{0.5em}
           \tiny (A concise and adaptable proof of Hardy)}
\author{Mohammad Bardestani}
\address{Mohammad Bardestani, Department of Mathematics and Statistics, University of Ottawa, 585 King Edward, Ottawa, ON K1N
6N5, Canada.}
\email{mbardest@uottawa.ca}
\author{Tristan Freiberg}
\address{Tristan Freiberg,Mathematics Department,
202 Mathematical Sciences Bldg,
University of Missouri,
Columbia, MO 65211 USA.}
\email{freibergt@missouri.edu}

\date{\today}

\begin{document}

\begin{abstract}
Myriad articles are devoted to Mertens's theorem.
In yet another, we merely wish to draw attention to a proof by 
Hardy, which uses a Tauberian theorem of Landau that ``leads 
to the conclusion in a direct and elegant manner''.
Hardy's proof is also quite adaptable, and it is readily combined 
with well-known results from prime number theory. 
We demonstrate this by proving a version of the theorem for primes 
in arithmetic progressions with uniformity in the modulus, as well 
as a non-abelian analogue of this. 
\end{abstract}

\maketitle
\tableofcontents

\section{Introduction}
 \label{sec:1}

In 1873, Mertens \cite{Mertens1874} proved that the difference
$
 \sum_{p \le x} p^{-1} - \log\log x
$
tends to a limit, $g = 0.26149\ldots$, as $x$ tends to infinity.
Actually, Mertens proved a somewhat sharper result than this by 
using an upper bound for $\pi(x)$ that had been established over 
20 years earlier by Chebyshev \cite{vcebyvsev1850memoire}.
(Cf.~\S\ref{sec:2} for notation.)
In the same work, Mertens extended his result to primes in 
arithmetic progressions: if $(q,a) = 1$ then 
\begin{align}
 \label{eq:1.1}
\sums[p \le x][p \equiv a \bmod q] p^{-1} 
 = \phi(q)^{-1}\log\log x + g(q,a) + \Oh[(\log x)^{-1}],
\end{align}
where $g(q,a)$ is a number depending on $q$ and $a$.
Naturally, Mertens borrowed from Dirichlet 
\cite{dirichlet1839beweis}, who in 1837 proved that, at least in 
the weak sense of analytic density, the primes are evenly 
distributed among the $\phi(q)$ possible arithmetic progressions to 
the modulus $q$.
(Dirichlet proved that as $x$ tends to infinity, the sum on the 
left-hand side of \eqref{eq:1.1} tends to 
$\phi(q)^{-1}\log\log x$.)

Mertens's extension of his result to arithmetic progressions is 
comparatively seldom used or cited.
It is referenced and proved by Landau in his 1909 {\em Handbuch} 
\cite[\S7, \S110]{MR0068565}.
In 1972, Williams \cite{MR0364137} gave a proof of \eqref{eq:1.1} 
in its product form, in which a form of the quantity corresponding 
to $g(q,a)$, which we will denote by $G(q,a)$, is given explicitly.
In 1975, Norton \cite{MR0419382} investigated a distribution 
related to the number of primes from an arbitrary set that 
divide an integer, and, in connection with the special case where 
the set is an arithmetic progression, the quantity $g(q,a)$. 
Bounds for $g(q,a)$, or more precisely, uniform bounds for
$\sum_{\text{$p \le x$, $p \equiv a \bmod q$}} p^{-1} 
 - \phi(q)^{-1}\log\log x$, 
were used by Rieger \cite{MR0319882} in 1972, and by Pomerance 
\cite{MR0447087, MR618552} in 1976/1980, to prove results on 
amicable numbers.
Much more recently, Languasco and Zaccagnini, in a series of papers 
starting with \cite{MR2351662}, investigated $g(q,a)$ and $G(q,a)$ 
extensively. 

The work of Languasco and Zaccagnini has the merit of being built 
up from ``first principles'' as regards zeros of Dirichlet 
$L$-functions.
They use the strongest results in this direction, including 
Siegel's theorem, to prove versions of \eqref{eq:1.1} that hold 
with uniformity in $q$. 
We will give a proof covering some of their results which, though 
requiring a Tauberian theorem of Landau, is easy to remember.
In the penultimate step, we obtain an exact expression into which 
one can simply ``plug'' familiar results or conjectures on the 
error term in the prime number theorem for arithmetic progressions, 
including conjectures that go beyond the generalized Riemann 
hypothesis for Dirichlet $L$-functions.

The proof is due to Hardy \cite{MR1574590}. 
Stating that the proof of Mertens's theorem in Landau's 
{\em Handbuch} is ``difficult to seize or to remember'', Hardy 
offered a proof that uses a ``well-known'' Tauberian theorem of 
Landau that ``leads to the conclusion in a direct and elegant 
manner''.
Hardy was not concerned with arithmetic progressions but his proof 
easily generalizes.
In fact, Hardy's proof is quite adaptable, and to illustrate this 
we will adapt it to the situation of primes whose Frobenius lies in 
a conjugacy class associated with a Galois number field. 

Primes that satisfy linear congruence conditions can be 
characterized in terms of their splitting behavior in an 
appropriate abelian number field, and {\em vice-versa}.
For instance, $p \equiv 1 \bmod q$ if and only of $p$ splits 
completely in the $q$-th cyclotomic field.
Results such as Mertens's theorem for arithmetic progressions are, 
in this sense, abelian, and one is often interested in their 
{\em non-abelian} analogues.
Hence we consider Mertens's theorem for primes that split 
completely in a general Galois number field, and so on.

We hope that this presentation will be useful to any reader who 
requires a version of Mertens's theorem for some other situation 
that is amenable to Hardy's proof.

This article is divided into two independent parts.
In Part I (\S\S\ref{sec:2}--\ref{sec:6}) we consider primes in 
arithmetic progressions, and in Part II 
(\S\S\ref{sec:7}--\ref{sec:10}) we consider primes whose Frobenius 
lies in a conjugacy class of automorphisms of a Galois number 
field.
The key propositions of each part are stated in \S\ref{sec:3} and 
\S\ref{sec:8}, their proofs relegated to \S\ref{sec:4} and 
\S\ref{sec:9}. 
In \S\ref{sec:5} and \S\ref{sec:10} we investigate the constants 
that arise ($g(q,a)$, $G(q,a)$ and their analogues).

\part*{%
\hspace*{\fill}%
I. PRIMES IN ARITHMETIC PROGRESSIONS%
\hspace*{\fill}%
}
\label{part:I}

\section{Background, notation and conventions}
 \label{sec:2}

We assume the reader is familiar with some standard results 
concerning the distribution of primes, the distribution of primes 
in arithmetic progressions, the Riemann zeta-function, Dirichlet 
characters and Dirichlet $L$-functions, such as can be found in 
\cite{MR1790423} or \cite{MR0417077}.

\begin{itemize}
  \item[] $p$ --- a prime number.
  \item[] $a,q$ --- positive integers, $a$ being coprime with $q$ 
          (written $(q,a) = 1$), and satisfying $a \le q$.
  \item[] $m,n,\nu$ --- positive integers.
  \item[] $t,x,y,\sigma,\tau$ --- real parameters.
          Unless stated otherwise, $t,x,y \ge 2$.
  \item[] $\epsilon, A$ --- real, positive numbers, where 
          $\epsilon$ can be arbitrarily small (but fixed), and $A$ 
          can be arbitrarily large (but fixed).
  \item[] $s$ --- a complex number with real part $\sigma$ and 
          imaginary part $\tau$, unless specified.
          The region denoted by
$
 \sigma \ge 1 - \frac{c}{\log(|\tau| + 3)}  
$
is to be understood as the region 
$
 \{s = \sigma + i\tau : 
  \sigma \ge 1 - {\textstyle\frac{c}{\log(|\tau| + 3)}}\},
$
and so on.
  \item[] $\gamma$ --- the Euler-Mascheroni constant:
\[
  \gamma 
  \defeq \lim_{n \to \infty} 
     \br{{\textstyle 1 + \frac{1}{2} + \ldots + \frac{1}{n}} 
        - \log n}
  = 0.57721\ldots.
\]
 \item[] $\phi(q)$ --- Euler's totient function.
 \item[] $\chi$, $\chi_0$ --- respectively a Dirichlet character to 
         the modulus $q$ and the principal Dirichlet character to 
         the modulus $q$.
 \item[] $\bar{\chi}$ --- the complex conjugate of $\chi$.
 \item[] $\zeta(s)$, $L(s,\chi)$ --- respectively the Riemann 
         zeta-function and the Dirichlet $L$-function corresponding 
         to $\chi$.
 \item[] $\log \zeta(s)$, $L(s,\chi)^{1/\phi(q)}$, etc. --- 
         single-valued branches of the corresponding functions that 
         are positive for $s = \sigma > 1$.
 \item[] $\pi(t)$, $\pi(t;q,a)$ --- the prime counting functions
\[
 \pi(t) \defeq \#\{p : p \le t\}, \quad
 \pi(t;q,a) \defeq \#\{p \le t : p \equiv a \bmod q\},
\]
$\#S$ denoting the cardinality of a set $S$.
 \item[] $\li{t}$ --- the logarithmic integral of $t$: 
\[
 \li{t} \defeq \int_2^t \frac{\dd{u}}{\log u}.
\]
 \item[] $E(t)$, $E(t;q,a)$ --- the error terms in the prime number 
         theorems:
\[
  E(t) \defeq \pi(t) - \li{t}, \quad 
  E(t;q,a) \defeq \pi(t;q,a) - \phi(q)^{-1}\li{t}.
\]
 \item[] $c$ --- an absolute positive constant, not necessarily the 
         same constant in each occurrence.
 \item[] $X = \Oh[Y]$, $X \ll Y$, $Y \gg X$ --- all mean that 
         $|X| \le C|Y|$ for some constant $C > 0$, which is 
         absolute unless indicated otherwise with subscripts, as in 
         $X \ll_{\epsilon, A} Y$, which means that $C$ depends on 
         $\epsilon$ and $A$.
 \item[] $X \asymp Y$ --- means that $X \ll Y \ll X$.
%  \item[] $f(x) \sim g(x)$ --- means that $f(x)/g(x) \to 1$ as 
%          $x \to \infty$.
 \item[] $\oh[1][s \to w]$ --- a quantity, depending on $s$, that 
         tends to $0$ in absolute value as $s$ tends to $w$.
\end{itemize}

\section{The main result and its corollaries}
 \label{sec:3}

Throughout, let $q$ and $a$ be integers such that 
$1 \le a \le q$ and $(q,a) = 1$.
Let $E(t;q,a) \defeq \pi(t;q,a) - \phi(q)^{-1}\li{t}$ 
denote the error term in the prime number theorem for arithmetic 
progressions. 
Let $\mathcal{L}(q,a)$ be the real, positive number satisfying 
\[
 \textstyle
  \mathcal{L}(q,a)^{\phi(q)} 
 \defeq
   \br{\frac{\phi(q)}{q}}
    \prod_{\chi \ne \chi_0} L(1,\chi)^{\bar{\chi}(a)},
\] 
and let 
\begin{align}
 \label{eq:3.1}
 g(q,a) 
 \defeq
  \phi(q)^{-1}\gamma 
   + \log \mathcal{L}(q,a) 
    - \underset{p^{\nu} \equiv a \bmod q}
               {\sum_{p}\sum_{\nu \ge 2}}
                 \frac{1}{\nu p^{\nu}}.
\end{align}

\begin{proposition}
 \label{prp:3.1}
Let $x \ge 2$ and let $q$ and $a$ be positive, coprime integers.  
We have
\begin{align}
 \label{eq:3.2}
  \int_2^{\infty} t^{-2} E(t;q,a) \dd{t}
 = 
    g(q,a)
   + \phi(q)^{-1}\log\log 2,
\end{align}
and
\begin{align}
 \label{eq:3.3}
  \begin{split}
  \sums[p \le x][p \equiv a \bmod q] p^{-1}
 & = \phi(q)^{-1}\log\log x + g(q,a) \\
 & \hspace{30pt}
   + x^{-1}  E(x;q,a) - \int_{x}^{\infty} t^{-2}E(t;q,a) \dd{t},
 \end{split}
\end{align}
with $g(q,a)$ as in \eqref{eq:3.1}.
\end{proposition}

We have tacitly used the fact that $L(1,\chi) \ne 0$ for 
$\chi \ne \chi_0$ in stating that $\mathcal{L}(q,a)$ is real and 
positive. 
This is tantamount to showing that there are infinitely many primes 
congruent to $a \bmod q$.
The only other result from prime number theory that we will use in 
the proof of Proposition \ref{prp:3.1} is the Chebyshev bound 
$
 \pi(t) \ll t(\log t)^{-1}.
$
Using this the reader may deduce the next proposition from the one
above.
(The reader may also verify Proposition \ref{prp:3.2} directly, in 
a manner similar to that of the proof of Proposition 
\ref{prp:3.1}.)

Let
\begin{align}
 \label{eq:3.4}
 \begin{split}
 G(q,a)
  & \defeq 
     \textstyle
      \exp
       \Br{%
      - \phi(q)^{-1}\gamma
       + g(q,a)
       +  \sum_{p \equiv a \bmod q}\sum_{\nu \ge 2}
           \frac{1}{\nu p^{\nu}} 
          } \\
  & \phantom{:}= 
  \mathcal{L}(q,a)\cdot
   \textstyle 
    \exp
     \Br{%
        \sum_{p \equiv a \bmod q}\sum_{\nu \ge 2}
         \frac{1}{\nu p^{\nu}}
        -
        \underset{p^{\nu} \equiv a \bmod q}
               {\sum_{p}\sum_{\nu \ge 2}}
                 \frac{1}{\nu p^{\nu}}
        }.
 \end{split}
\end{align}

\begin{proposition}
 \label{prp:3.2}
Let $x \ge 2$ and let $q$ and $a$ be positive, coprime integers.
We have
\begin{align}
 \label{eq:3.5}
 \begin{split}
 & \sums[p \le x][p \equiv a \bmod q] 
    \log\br{1 - \frac{1}{p}}^{-1} \\ 
 & \hspace{60pt} =
    \phi(q)^{-1}(\gamma + \log\log x)  
   + \log G(q,a) \\
 & \hspace{90pt}
  + x^{-1}E(x;q,a) 
   - \int_{x}^{\infty} t^{-2}E(t;q,a) \dd{t}
    + \Oh[(x\log x)^{-1}],
 \end{split}
\end{align}
with $G(q,a)$ as in \eqref{eq:3.4}.
The implicit constant is absolute.
\end{proposition}

We have the following immediate corollary to the above 
propositions. 

\begin{theorem}
 \label{thm:3.3}
Let $q$ and $a$ be positive, coprime integers.
We have
\begin{align}
 \label{eq:3.6}
 \sums[p \le x][p \equiv a \bmod q] p^{-1}
  =
  \phi(q)^{-1}\log\log x + g(q,a) + \oh[1][x \to \infty],
\end{align}
and
\begin{align}
 \label{eq:3.7}
 \prods[p \le x][p \equiv a \bmod q]\br{1 - \frac{1}{p}}^{-1}
  =
   G(q,a)\cdot 
    \br{\e^{\gamma} \log x}^{1/\phi(q)}
     \cdot\Br{1 + \oh[1][x \to \infty]}.
\end{align}
Here, $g(q,a)$ and $G(q,a)$ are as in \eqref{eq:3.1} and 
\eqref{eq:3.4} respectively.
\end{theorem}

For fixed $q$, the prime number theorem for arithmetic progressions 
states that $E(t;q,a) \ll t\e^{-c\sqrt{\log t}}$. 
(Cf.~\cite[Th\'eor\`eme 8.5]{MR0417077}.)
Thus, for fixed $q$, equality holds in \eqref{eq:3.6} with an 
error of $\Oh[\e^{-c\sqrt{\log x}}]$.
Similarly for \eqref{eq:3.7}.
We get an error of  
$
\Oh[\e^{-c(\log x)^{3/5}(\log\log x)^{-1/5}}] 
$
if we use the Korobov-Vinogradov bound for $E(t;q,a)$.
In fact, by the prime number theorem for arithmetic progressions, 
these estimates hold uniformly for all non-exceptional $q$.
(The product $\prod_{\chi} L(s,\chi)$ has at most one zero in the 
region $\sigma \ge 1 - \frac{1}{100\log(q(\ab{\tau}+3))}$.
We say $q$ is {\em exceptional} if the product does have a zero in 
this region.)

Conjecturally there are no exceptional moduli, but in case there 
are we have to use Siegel's theorem to obtain uniform results with 
an error term as good as in the prime number theorem for arithmetic 
progressions. 
The Siegel-Walfisz theorem states that for any given number 
$A > 0$, we have $E(t;q,a) \ll t\e^{-c(A)\sqrt{\log t}}$, uniformly 
for $q \le (\log t)^A$, where $c(A)$ is a positive constant 
depending on $A$.
We also have $E(t;q,a) \ll_A t(\log t)^{-A}$, uniformly for all 
$q$. 
(Cf.~\cite[Th\'eor\`eme 8.5]{MR0417077}).
Combining with Propositions \ref{prp:3.1} and \ref{prp:3.2} gives 
the following result.

\begin{theorem}
 \label{thm:3.4}
Fix a number $A > 0$.
(a) We have
\begin{align}
 \label{eq:3.8}
  \sums[p \le x][p \equiv a \bmod q] p^{-1}
 =
  \phi(q)^{-1}\log\log x 
 + g(q,a)
  + \Oh[\e^{-c(A)\sqrt{\log x}}],
\end{align}
and
\begin{align}
 \label{eq:3.9}
 \prods[p \le x][p \equiv a \bmod q] \br{1 - \frac{1}{p}}^{-1}
 =
 G(q,a)\cdot (\e^{\gamma}\log x)^{1/\phi(q)}\cdot
  \Br{1 + \Oh[\e^{-c(A)\sqrt{\log x}}]},
\end{align}
uniformly for $x \ge 2$, $1 \le q \le (\log x)^A$, and $a \ge 1$ 
with $(q,a) = 1$, where $c(A)$ is a positive constant that 
depends only on $A$.
(The implicit constants are absolute.)
Here, $g(q,a)$ and $G(q,a)$ are as in \eqref{eq:3.1} and 
\eqref{eq:3.4} respectively.
(b) We have
\begin{align}
 \label{eq:3.10}
  \sums[p \le x][p \equiv a \bmod q] p^{-1}
 =
  \phi(q)^{-1}\log\log x 
 + g(q,a)
  + \Oh[(\log x)^{-A}][A],
\end{align}
and
\begin{align}
 \label{eq:3.11}
 \prods[p \le x][p \equiv a \bmod q] \br{1 - \frac{1}{p}}^{-1}
 =
 G(q,a)\cdot (\e^{\gamma}\log x)^{1/\phi(q)}\cdot
  \Br{1 + \Oh[(\log x)^{-A}][A]},
\end{align}
uniformly for $x \ge 2$, $q \ge 1$, and $a \ge 1$ with 
$(q,a) = 1$.
The implicit constants depend only on $A$.
Here, $g(q,a)$ and $G(q,a)$ are as in \eqref{eq:3.1} and 
\eqref{eq:3.4} respectively.
\end{theorem}

We can replace the $\Oh$-terms in \eqref{eq:3.8} and 
\eqref{eq:3.9} by 
$
 \Oh[\e^{-c(A)(\log x)^{3/5}(\log\log x)^{-1/5}}], 
$
by using the Korobov-Vinogradov bound for $E(t;q,a)$.

If the generalized Riemann hypothesis for Dirichlet 
$L$-functions (GRH) holds then $E(t;q,a) \ll t^{1/2}\log t$, 
uniformly for all $q$.
Beyond even GRH, H.~Montgomery has conjectured that for any 
$\epsilon > 0$, $E(t;q,a) \ll_{\epsilon} q^{-1/2}t^{1/2+\epsilon}$, 
for any $q$.
Hence the following conditional results. 

\begin{theorem}
 \label{thm:3.5}
On GRH we have
\begin{align}
 \label{eq:3.12}
  \sums[p \le x][p \equiv a \bmod q] p^{-1}
 =
  \phi(q)^{-1}\log\log x 
 + g(q,a)
  + \Oh[x^{-1/2}\log x],
\end{align}
and
\begin{align}
 \label{eq:3.13}
  \begin{split}
 \prods[p \le x][p \equiv a \bmod q] \br{1 - \frac{1}{p}}^{-1}
 & =
 G(q,a)\cdot (\e^{\gamma}\log x)^{1/\phi(q)}\cdot
  \Br{1 + \Oh[x^{-1/2}\log x]},
  \end{split}
\end{align}
uniformly for $x \ge 2$, $q \ge 1$, and $a \ge 1$ with $(q,a) = 1$.
The implicit constants are absolute.
Here, $g(q,a)$ and $G(q,a)$ are as in \eqref{eq:3.1} and 
\eqref{eq:3.4} respectively.
\end{theorem}

\begin{theorem}
 \label{thm:3.6}
Fix any number $\epsilon > 0$.
Let $x \ge 2$ be a number and let $q$ and $a$ be positive, coprime 
integers.
If Montgomery's conjecture holds, then
\begin{align}
 \label{eq:3.14}
  \sums[p \le x][p \equiv a \bmod q] p^{-1}
 =
  \phi(q)^{-1}\log\log x 
 + g(q,a)
  + \Oh[q^{-1/2}x^{-1/2 + \epsilon}][\epsilon],
\end{align}
and
\begin{align}
 \label{eq:3.15}
  \begin{split}
 \prods[p \le x][p \equiv a \bmod q] \br{1 - \frac{1}{p}}^{-1}
 & =
 G(q,a)\cdot (\e^{\gamma}\log x)^{1/\phi(q)}\cdot 
  \Br{1 + \Oh[q^{-1/2}x^{-1/2 + \epsilon}][\epsilon]}.
  \end{split}
\end{align}
The implicit constants depend only on $\epsilon$.
Here, $g(q,a)$ and $G(q,a)$ are as in \eqref{eq:3.1} and 
\eqref{eq:3.4} respectively.
\end{theorem}

\section{Proof of the main result}
 \label{sec:4}

In addition to the bound $\pi(t) \ll t(\log t)^{-1}$, we will use 
the following results in the proof of Proposition \ref{prp:3.1}. 
The first lemma is a standard result in the theory of the Gamma 
function. 
The second lemma is a Tauberian theorem due to Landau 
\cite{44.0282.02}.

\begin{lemma}
 \label{lem:4.1}
Let $\eta > 0$ and $\delta > 0$ be given. 
We have
\[
 \int_{\eta}^{\infty} u^{-1}\e^{-\delta u} \dd{u} 
 = 
  \textstyle 
   \log \frac{1}{\delta} 
  - \log \eta 
   - \gamma 
    + \oh[1][\delta \to 0].
\]
\end{lemma}

\begin{lemma}
 \label{lem:4.2}
Let $\eta > 0$ and $\delta > 0$ be given, and let 
$f$ be a real-valued integrable function on $[\eta,\infty)$. 
Suppose that the integral 
\[
 J(\delta) \defeq \int_{\eta}^{\infty} f(u)u^{-\delta} \dd{u} 
\]
converges and tends to a limit $L$ as $\delta \to 0$, and 
suppose that $f(u) \ll (u\log u)^{-1}$ for $u \ge \eta$. 
Then $J(0)$ converges and is equal to $L$. 
\end{lemma}

\begin{proof}[Proof of Proposition \ref{prp:3.1}] 
The functions
\begin{align*}
  \mathcal{L}(s;q,a)
 & \defeq \textstyle
    \zeta(s)^{-1/\phi(q)}
     \prod_{\chi} L(s,\chi)^{\bar{\chi}(a)/\phi(q)} \\
 &\phantom{:}= \textstyle
   \prod_{p \mid q}\br{1 - \frac{1}{p^s}}^{1/\phi(q)}
    \prod_{\chi \ne \chi_0} L(s,\chi)^{\bar{\chi}(a)/\phi(q)}
\end{align*}
and
\[
 g(s;q,a) 
  \defeq 
   \phi(q)^{-1}\gamma
   + \log \mathcal{L}(s;q,a)
    - \underset{p^{\nu} \equiv a \bmod q}
               {\sum_{p}\sum_{\nu \ge 2}}
                \frac{1}{\nu p^{\nu s}}
\]
are defined and analytic throughout the region where 
$\sigma > \frac{1}{2}$ and $\prod_{\chi \ne \chi_0} L(s,\chi)$ is 
non-zero.
In particular, this region contains the half-plane $\sigma \ge 1$.
Note that as $s \to 1$, $\mathcal{L}(s;q,a) \to \mathcal{L}(q,a)$ 
and hence $g(s;q,a) \to g(q,a)$.

Using the orthogonality relation 
\begin{align}
 \label{eq:4.1}
  \begin{split}
 \sum_{\chi} \bar{\chi}(m)\chi(n)
 = 
\begin{cases}
  \phi(q) & \text{if $(q,m) = 1$ and $m \equiv n \bmod q$} \\
  0       & \text{otherwise,}
\end{cases}
 \end{split}
\end{align}
it is straightforward to verify that for $\sigma > 1$, 
\[
  \sums[p \equiv a \bmod q] p^{-s}
 = 
     \phi(q)^{-1}\br{\log \zeta(s) - \gamma} 
   +  g(s;q,a).  
\]
On the other hand, for $\sigma > 1$, partial summation yields
\[
 \sums[p \le y][p \equiv a \bmod q] p^{-\sigma}
 = y^{-\sigma}\pi(y;q,a) 
  + \phi(q)^{-1}\sigma 
     \int_2^y \frac{\li{t}}{t^{1+\sigma}} \dd{t}
  + \sigma \int_2^y \frac{E(t;q,a)}{t^{1+\sigma}} \dd{t},
\]
and since $\pi(t;q,a), |E(t;q,a)|, \li{t} < t$, letting $y$ tend 
to infinity yields
\[
 \sum_{p \equiv a \bmod q} p^{-\sigma}
 = 
  \phi(q)^{-1}\sigma 
   \int_2^{\infty} \frac{\li{t}}{t^{1+\sigma}} \dd{t}
  + \sigma \int_2^{\infty} \frac{E(t;q,a)}{t^{1+\sigma}} \dd{t}.
\]
Integration by parts, followed by the substitution $u = \log t$, 
followed by an application of Lemma \ref{lem:4.1}, followed by the 
Laurent series for $\zeta(s)$ about $s = 1$, which gives  
\begin{align}
 \label{eq:4.2}
 \zeta(s) = \frac{1}{s-1} + \gamma + \Oh[s - 1], 
  \quad 0 < |s-1| < 1,
\end{align}
yields
\begin{align*}
    \phi(q)^{-1}\sigma 
   \int_{2}^{\infty} \frac{\li{t}}{t^{1+\sigma}} \dd{t} 
& = 
   \phi(q)^{-1}  
   \int_{2}^{\infty} \frac{\dd{t}}{t^{\sigma}\log t}  
  = 
    \phi(q)^{-1}
  \int_{\log 2}^{\infty} u^{-1}\e^{-(\sigma - 1)u} \dd{u} \\
 & = 
   \textstyle
   \phi(q)^{-1} 
   \br{%
    \log{\br{\frac{1}{\sigma - 1}}} 
   - \gamma 
    - \log\log 2 
     + \oh[1][\sigma \to 1]%
     } \\
  & = 
  \textstyle
   \phi(q)^{-1}  
   \br{%
    \log{\zeta(\sigma)} 
   - \gamma 
    - \log\log 2 
     + \oh[1][\sigma \to 1]%
     }.
\end{align*}
Combining all of this, we obtain
\[
 \sigma \int_2^{\infty} \frac{E(t;q,a)}{t^{1+\sigma}} \dd{t}
 =  
   g(\sigma;q,a)  
  + \phi(q)^{-1}  
     \br{\log\log 2 
        + \oh[1][\sigma \to 1]}.
\]
Since $\pi(t;q,a), \li{t} \ll t(\log t)^{-1}$, we have
$t^{-2}E(t;q,a) \ll (t\log t)^{-1}$.
Therefore, by Lemma \ref{lem:4.2}, 
$\int_2^{\infty} t^{-2}E(t;q,a)\dd{t}$ converges, and 
\eqref{eq:3.2} follows.

Partial summation yields
\[
 \sums[p \le x][p \equiv a \bmod q] p^{-1}
  = 
   x^{-1}\pi(x;q,a) 
   + \phi(q)^{-1} 
      \int_2^x t^{-2}\li{t} \dd{t} 
     + \int_2^x t^{-2}E(t;q,a) \dd{t}.
\]
We have
$  x^{-1}\pi(x;q,a) 
 = \phi(q)^{-1} x^{-1}\li{x}
    + x^{-1} E(x;q,a);
$
integration by parts yields
\[
 \phi(q)^{-1} 
      \int_2^x t^{-2}\li{t} \dd{t}
 = 
   \phi(q)^{-1} 
    \br{-x^{-1}\li{x} + \log\log x - \log\log 2};
\]
and by \eqref{eq:3.2},
\[
 \int_2^x t^{-2}E(t;q,a) \dd{t}
 = 
  g(q,a) 
 + \phi(q)^{-1} \log\log 2
 - \int_x^{\infty} t^{-2}E(t;q,a) \dd{t}.
\]
Combining yields \eqref{eq:3.3}.
\end{proof}

\section{More about $g(q,a)$ and $G(q,a)$}
 \label{sec:5}

The main purpose of this section is to estimate $g(q,a)$ and 
$G(q,a)$, as defined in \eqref{eq:3.1} and \eqref{eq:3.4}.
The results here are more or less contained in one of Norton 
\cite[Lemma 6.3]{MR0419382}, and the proofs are similar. 
(Also see \cite[Theorem 1]{MR0447087}.)
The results of this section should be borne in mind when 
considering the range of uniformity in $q$ for which the results of 
\S \ref{sec:3} hold.
That is, for large $q$, the sums and products we are interested in 
are dominated by their first few terms, and estimates can become 
almost trivial. 

For instance, the main result of Languasco and Zaccagnini 
\cite[Theorem 2]{MR2351662}, \linebreak 
as it directly and explicitly accounts for the effect of a putative 
Siegel zero on multiples of an exceptional modulus in a given 
range, is rather precise and all-encompassing, although it takes 
some digesting.
Of course, as noted by Languasco and Zaccagnini, from their main 
result one may use Siegel's theorem to deduce our Theorem 
\ref{thm:3.4} (with corresponding Korobov-Vinogradov $\Oh$-terms).
However, they also offer the following statement 
\cite[Corollary 3]{MR2351662}: 
\[
 \prods[p \le x][p \equiv a \bmod q]\br{1 - \frac{1}{p}}^{-1}
  = G(q,a) \cdot (\e^{\gamma}\log x)^{1/\phi(q)}
            \cdot 
             \Br{1 
            + \Oh[\frac{(\log\log x)^{16/5}}{(\log x)^{3/5}}][A]}, 
\]
uniformly for $q \le \exp\Br{A(\log x)^{2/5}(\log\log x)^{1/5}}$ 
that are multiples of an exceptional modulus in the same range.
(Of course, this estimate holds for all $q$ in that range [indeed 
we have \eqref{eq:3.11} for {\em all} $q$], but for $q$ that are 
not multiples of an exceptional modulus we have the more precise 
estimate \eqref{eq:3.9}.)
But for {\em all} $q \le (\log x)^2$, say, we have \eqref{eq:3.9}, 
while for all $q > (\log x)^2$, a sharper estimate can be 
obtained without using any information about primes (see below).

Before proceeding, note that by using the orthogonality relations 
\eqref{eq:4.1} and
\begin{align}
 \label{eq:5.1}
 \begin{split}
 \sum_{1 \le m \le q} \bar{\chi}(m)
 = 
\begin{cases}
  \phi(q) & \text{if $\chi = \chi_0$} \\
  0       & \text{otherwise,}
\end{cases}
 \end{split}
\end{align}
it is straightforward to verify that 
\[
\sums[a \bmod q][(q,a) = 1] g(q,a)
 = g 
  - \sum_{p \mid q} p^{-1},
\]
where
\[
g 
 \defeq g(1,1)
  = \gamma - \sum_{p}\sum_{\nu \ge 2} \frac{1}{\nu p^\nu}
  = 0.26149\ldots
\]
is Mertens's constant.
We can use this to check consistency, for
\[
 E(t) 
  = 
   \sums[a \bmod q][(q,a) = 1] E(t;q,a)
  + \omega(t;q),
\]
where $\omega(t;q) \defeq \#\{p \le t : p \mid q\}$,
and putting this into \eqref{eq:3.1} gives
\begin{align*}
  \int_2^{\infty} t^{-2}E(t) \dd{t}
 & =
   \sums[a \bmod q][(q,a) = 1] g(q,a)
  + \log\log 2
   + \int_2^{\infty} t^{-2} \omega(t;q) \dd{t} \\
 & =    
    \sums[a \bmod q][(q,a) = 1] g(q,a)
   + \log\log 2
    + \sum_{p \mid q} p^{-1}.
\end{align*}
(To see the last equality, note that
\[
 \sum_{\text{$p \le y$ : $p \mid q$}} p^{-1}
  = y^{-1}\omega(y;q) 
   + \int_2^y t^{-2} \omega(t;q) \dd{t}
\]
by partial summation, and let $y$ tend to infinity.)
On the other hand, putting $q = 1$ into \eqref{eq:3.1} gives 
$\int_2^{\infty} t^{-2}E(t) \dd{t} = g + \log\log 2$. 

Recall that $q$ and $a$ are supposed to be positive integers such 
that $(q,a) = 1$ and $a \le q$.
Let
\begin{align}
 \label{eq:5.2}
  \begin{split}
 g^{*}(q,a)
  \defeq
  \begin{cases}
   a^{-1}  & \text{if $q \ge 2$ and $a$ is prime} \\
   0         & \text{otherwise,}
  \end{cases}
  \end{split}
\end{align}
and let
\begin{align}
 \label{eq:5.3}
  \begin{split}
 G^{*}(q,a)
  \defeq
  \begin{cases}
   \br{1 - \frac{1}{a}}^{-1}   
      & \text{if $q \ge 2$ and $a$ is prime} \\
   1                          
      & \text{otherwise.}
  \end{cases}
  \end{split}
\end{align}
Elementarily, for $q < x$ we have
\[
 \sums[q < p \le x][p \equiv a \bmod q] p^{-1}
  \le \sum_{n \le x/q} (qn)^{-1}
   \ll q^{-1}\log(x/q),
\]
hence
\begin{align*}
 \sums[p \le x][p \equiv a \bmod q] p^{-1}
  = 
 g^{*}(q,a) + 
     \begin{cases}
      \Oh\br{q^{-1}\log(x/q)} & \text{if $q < x$} \\
       0                      & \text{if $q \ge x$,}
    \end{cases}
\end{align*}
and similarly,
\begin{align*}
 G^{*}(q,a)
  \le 
   \prods[p \le x][p \equiv a \bmod q]\br{1 - \frac{1}{p}}^{-1}
    \le
 G^{*}(q,a)\times
    \begin{cases}
      \exp\br{cq^{-1}\log(x/q)} & \text{if $q < x$} \\
      1                         & \text{if $q \ge x$.}
    \end{cases}
\end{align*}
Thus, for $c\log(x/q) < q < x$ we have
\[
 \prods[p \le x][p \equiv a \bmod q]\br{1 - \frac{1}{p}}^{-1}
  = G^{*}(q,a)\cdot \Br{1 + \Oh[q^{-1}\log(x/q)]}.
\]
This is entirely elementary and uses no information about primes.

Using the Brun-Titchmarsh inequality, which states that for 
$q < t$,  
\[
  \pi(t;q,a) \ll \frac{t}{\phi(q)\log(3t/q)},
\]
and partial summation, it is straightforward to show that for 
$q < x$,  
\[
 \sums[q < p \le x][p \equiv a \bmod q] p^{-1}
  \ll \phi(q)^{-1}\log\log(3x/q),
\]
from which we can deduce the following result.
\begin{theorem}
 \label{thm:5.1}
Let $x \ge 2$ and let $q$ and $a$ be positive, coprime integers 
with $a \le q$.
We have
\begin{align}
 \label{eq:5.4}
 \sums[p \le x][p \equiv a \bmod q] p^{-1}
  = 
   g^{*}(q,a) +
    \begin{cases}
      \Oh\br{\phi(q)^{-1}\log\log(3x/q)} & \text{if $q < x$} \\
       0                                 & \text{if $q \ge x$,}
    \end{cases}
\end{align}
where the implicit constant is absolute, and
\begin{align}
 \label{eq:5.5}
  \begin{split}
  G^{*}(q,a) 
   & \le
    \prods[p \le x][p \equiv a \bmod q]\br{1 - \frac{1}{p}}^{-1} \\
   & \hspace{45pt} \le 
 G^{*}(q,a) \times 
  \begin{cases}
   \exp\br{c\phi(q)^{-1}\log\log(3x/q)} 
        & \text{if $q < x$} \\
      1                                      
        & \text{if $q \ge x$,}
  \end{cases}
  \end{split}
\end{align}
where $c > 0$ is an absolute constant. 
Here, $g^{*}(q,a)$ and $G^{*}(q,a)$ are as in \eqref{eq:5.2} and 
\eqref{eq:5.3} respectively.
\end{theorem}

Thus, if \eqref{eq:3.4} and \eqref{eq:3.5} are to 
hold with $\phi(q)$ growing faster than $\log\log x$, for a given 
$a$ we must have $g(q,a) \to g^{*}(q,a)$ and 
$G(q,a) \to G^{*}(q,a)$ as $q$ tends to infinity over integers that 
are coprime with $a$. 

\begin{theorem}
 \label{thm:5.2}
Let $q$ and $a$ be positive, coprime integers.
Let $g(q,a)$, $g^{*}(q,a)$, $G(q,a)$ and $G^{*}(q,a)$ be as in 
\eqref{eq:3.1}, \eqref{eq:5.2}, \eqref{eq:3.4} and \eqref{eq:5.3} 
respectively.
We have
\begin{align}
 \label{eq:5.6}
 g(q,a) = g^{*}(q,a) +  \Oh[\phi(q)^{-1}\log q],
\end{align}
and
\begin{align}
 \label{eq:5.7}
 G(q,a) = G^{*}(q,a)\cdot \Br{1 + \Oh[\phi(q)^{-1}\log q]}.
\end{align}
The implicit constants are absolute.  
\end{theorem}
\begin{proof}
Suppose, as we may, that $q \ge 3$. 
By \eqref{eq:3.1} we have
\[
 g(q,a)
  =
  \Br{\int_2^q + \int_q^{\exp(q)} 
   + \int_{\exp(q)}^{\infty}} t^{-2}E(t;q,a) \dd{t} 
   - \phi(q)^{-1}\log\log 2.
\]
If $a$ is not prime, then
\begin{align*}
 \int_2^q t^{-2}E(t;q,a) \dd{t}
 & =
  -\phi(q)^{-1}\int_2^q t^{-2}\li{t} \dd{t} \\
 & = 
   -\phi(q)^{-1}
    \br{-q^{-1}\li{q} + \log\log q - \log\log 2},
\end{align*}
whereas if $a$ is prime, the left-hand side is equal to this 
plus
\[
 \int_a^q t^{-2} \dd{t} = a^{-1} - q^{-1}.
\]
Whether or not $a$ is prime, by the Brun-Titchmarsh inequality 
we have
\[
 \int_q^{\exp(q)} t^{-2}E(t;q,a) \dd{t}
  \ll \phi(q)^{-1}
       \int_q^{\exp(q)} (t\log(3t/q))^{-1} \dd{t}
    \ll \phi(q)^{-1}\log q.
\]
Also, by the Siegel-Walfisz theorem we have
\[
 \int_{\exp(q)}^{\infty} t^{-2}E(t;q,a) \dd{t}
  \ll \int_{\exp(q)}^{\infty} t^{-1}\e^{-c\sqrt{\log t}} \dd{t}
   \ll \e^{-c\sqrt{q}}.
\]
Combining gives the result for $g(q,a)$.

By putting this into the definition of $G(q,a)$, we 
deduce the result for $G(q,a)$ by noting that
\[
 \sum_{p \equiv a \bmod q} 
  \sum_{\nu \ge 2} \frac{1}{\nu p^{\nu}} 
 = \sum_{p \equiv a \bmod q} 
    \textstyle
     \br{\log\br{1 - \frac{1}{p}}^{-1} - \frac{1}{p}},
\] 
and that 
\[
 \sum_{p > q} \sum_{\nu \ge 2} \frac{1}{\nu p^{\nu}} 
  \le \sum_{n > q} n^{-2}
   \ll q^{-1}.
\]
\end{proof}

Two remarks.
First, we do not use the full strength of the Siegel-Walfisz 
theorem in the above proof, and the implicit constants involved are 
effectively computable.
Second, if we assume the aforementioned conjecture of Montgomery, 
GRH, or just that $q$ is not an exceptional modulus, 
Theorem \ref{thm:5.2} can be improved slightly.
Namely, as the reader may verify, the $\Oh$-terms in 
\eqref{eq:5.6} and \eqref{eq:5.7} can be replaced 
by $\Oh[\phi(q)^{-1}\log\log q]$ (for $q \ge 3$).

We end by noting the following interpretation of the constant 
$g(q,a)$ in the special case where $q \ge 3$ and 
$a \equiv 1 \bmod q$.
Recall that
\[
 \mathcal{L}(q,1)^{\phi(q)} \textstyle
  \defeq \br{\frac{\phi(q)}{q}}
          \prod_{\chi \ne \chi_0} L(1,\chi)
 = \lim_{s \to 1} \zeta(s)^{-1}\prod_{\chi} L(s,\chi).
\]
The right-hand side is equal to $\Res{s=1} \zeta_L(s)$, where 
$\zeta_L(s)$ is the Dedekind zeta-function of the $q$-th cyclotomic 
field $L \defeq \Q(\e^{2\pi i/q})$.
(Cf.~\S\ref{sec:10}.)
By the {\em analytic class number formula} for $L$, 
\[
 \Res{s=1} \zeta_L(s)
  = \frac{2^{r_1}(2\pi)^{r_2}h_L\Reg{L}}{w_L\sqrt{|\Delta_L|}}, 
\]
where: $\dgr{L} = r_1 + 2r_2$, $r_1$ and $2r_2$ respectively 
denoting the number of real and complex embeddings of $L$;
$h_L$ is the {\em class number} of $L$;
$\Reg{L}$ is the {\em regulator} of $L$; and
$w_L$ is the number of roots of unity in $L$.
Recalling the definition \eqref{eq:3.1} of $g(q,1)$, then 
\eqref{eq:3.2} and \eqref{eq:5.6}, we have shown that 
\begin{align*}
 \textstyle\frac{1}{\phi(q)}
  \log\br{\frac{2^{r_1}(2\pi)^{r_2}h_L\Reg{L}}
               {w_L\sqrt{|\Delta_L|}}}
 & = g(q,1) - \phi(q)^{-1}\gamma 
    + \underset{p^{\nu} \equiv 1 \bmod q}
               {\sum_{p}\sum_{\nu \ge 2}}
                 \frac{1}{\nu p^{\nu}} \\
 & = \int_2^{\infty} t^{-2}E(t;q,1) \dd{t}
      - \phi(q)^{-1}(\gamma + \log\log 2) \\
 & \hspace{120pt}
       + \underset{p^{\nu} \equiv 1 \bmod q}
               {\sum_{p}\sum_{\nu \ge 2}}
                 \frac{1}{\nu p^{\nu}} \\
 & = \underset{p^{\nu} \equiv 1 \bmod q}
               {\sum_{p}\sum_{\nu \ge 2}}
                 \frac{1}{\nu p^{\nu}} 
     + \Oh[\phi(q)^{-1}\log q].
\end{align*}
As remarked following the proof of Theorem \ref{thm:5.2}, the 
$\Oh$-term here can be replaced by $\Oh[\phi(q)^{-1}\log\log q]$ 
for non-exceptional moduli $q$.
See \S\ref{sec:10} for an extension of this.

\section{Integers composed of primes in an arithmetic progression}
 \label{sec:6}

In this section we digress slightly to estimate a quantity closely 
related to Mertens's theorem for primes in an arithmetic 
progression, namely the number of integers up to $x$ that are 
composed only of primes in a given arithmetic progression.
We add the ``twist'' that the primes must also exceed a number $y$, 
$1 \le y \le (\log x)^A$. 
 
For numbers $x,y \ge 1$ and positive, coprime integers $q,a$, let
\[
 S(y) 
 \defeq 
  \{n \in \mathbb{N} : p \mid n \Rightarrow p > y\}, 
   \quad 
 S(q,a) 
  \defeq
   \{n \in \mathbb{N} : p \mid n \Rightarrow p \equiv a \bmod q\},
\]
let
\[
 \Phi(x,y) \defeq \#\{n \le x : n \in S(y)\}, \quad
 \Phi(x;q,a) \defeq \#\{n \le x : n \in S(q,a)\},
\]
and let
\[
\Phi(x,y;q,a) \defeq \#\{n \le x : n \in S(y) \cap S(q,a)\}.
\]
Recall that for $\sigma > 0$, the improper integral 
\[
 \Gamma(s) \defeq \int_0^{\infty} \e^{-u}u^{s - 1} \dd{s}
\]
converges.
The Gamma function is obtained by extending this integral function 
analytically to all $s \ne 0, -1, -2, \ldots$ (the non-positive 
integers are simple poles).

In the following theorem, part (a) is a generalization (to $y > 1$) 
of an estimate proved in \cite[p.125]{MR1397501}, and the proof is 
the same.
Part (b) is proved in the same way, the only difference being that 
we keep track of the dependence on $q$ in our estimates.

\begin{theorem}
 \label{thm:6.1}
Fix $A > 0$. 
Let $x$ and $y$ be numbers satisfying $1 \le y \le (\log x)^A$.
Let $q$ and $a$ be positive, coprime integers, and let $G(q,a)$ be 
as in \eqref{eq:3.4}. 
(a) For any $\epsilon > 0$ we have, uniformly for $y$ in the given 
range and $q \ge 1$, 
\begin{align}
 \label{eq:6.1}
  \begin{split}
 \Phi(x,y;q,a)
  & = x(\log x)^{1/\phi(q) - 1}  \\
  & \hspace{30pt} \times\textstyle
     \Br{\frac{G(q,a)}{\Gamma(1/\phi(q))}
      \prods[p \le y][p \equiv a \bmod q]\br{1 - \frac{1}{p}}
     + \Oh[(\log x)^{\epsilon-1}][\epsilon,A]}.
 \end{split}
\end{align}
The implicit constant depends on $\epsilon$ and $A$ at most.
(b) For any $\epsilon > 0$ we have, for $y$ in the given range and 
$1 \le q \le (\log x)^A$,
\begin{align}
 \label{eq:6.2}
  \begin{split}
 \Phi(x,y;q,a)
  & = x(\log x)^{1/\phi(q) - 1}  \\
  & \hspace{30pt} \times\textstyle
     \Br{\frac{G(q,a)}{\Gamma(1/\phi(q))}
      \prods[p \le y][p \equiv a \bmod q]\br{1 - \frac{1}{p}}
     + \Oh[\frac{ q^{\epsilon} (\log y)^{1 + 2/\phi(q)} }
                {\log x}][\epsilon,A]}.
 \end{split}
\end{align}
The implicit constant depends on $\epsilon$ and $A$ at most.
(c) We have, uniformly for $y$ in the given range and 
$1 \le q \le (\log\log x)^A$, 
\begin{align}
 \label{eq:6.3}
  \begin{split}
 \Phi(x,y;q,a)
  & = \textstyle
   x(\log x)^{1/\phi(q) - 1} 
     \frac{G(q,a)}{\Gamma(1/\phi(q))}  
     \prods[p \le y][p \equiv a \bmod q]\br{1 - \frac{1}{p}} \\
  & \hspace{120pt} \times\textstyle
    \Br{1 + \Oh[\frac{(\log\log x)^{A+3}}{\log x}][A] }.
 \end{split}
\end{align}
The implicit constant depends on $A$ at most.
\end{theorem}

\begin{proof}
(a) Let $\Omega(n)$ be the number of prime divisors of $n$, counted 
with multiplicity.
Note that for all $x \ge 3$, $y \ge 1$, $q \ge 1$, and any number 
$C > 0$, 
\[
\Phi(x,y;q,a)
 \le \Phi(x;q,a)
  \le \sum_{\substack{n \in [1,x] \cap S(q,a) 
              \\ \Omega(n) \le C\log\log x}} 1
      + \sums[n \le x][\Omega(n) > C\log\log x] 1.
\]
Now if $n \in S(q,a)$ and $\Omega(n) \le C\log\log x$, then we have
$n \equiv a^{\nu} \bmod q$ for some $\nu \le C\log\log x$, and by 
the Hardy-Rananujan inequality 
(cf.~\cite[III, Exercise 3.1(d)]{MR1342300}), for any given 
constant $B > 0$, we may choose $C = C(B)$ so that 
$\#\{n \le x : \Omega(n) > C\log\log x\} \ll_B x(\log x)^{-B}$.
Thus,  
\[
  \Phi(x;q,a)
   \ll_B  \frac{x\log\log x}{q} + \frac{x}{(\log x)^B}.
\]
If $q \gg_B (\log x)^{B+1}$, we have
$
 \Phi(x;q,a) \ll_{B} x(\log x)^{-B}. 
$

On the other hand, since $\Gamma(1/\phi(q)) \asymp \phi(q)$, and 
since $G(q,a) \asymp 1$ by \eqref{eq:5.7}, \eqref{eq:6.1} is 
equivalent, for $q \gg_B (\log x)^{B+1}$, to the weaker bound
\[
 \Phi(x,y;q,a) 
  \ll_{\epsilon,A,B} x(\log x)^{\epsilon - 2}.
\]
Thus we may assume, without loss of generality, that 
$q \le (\log x)^A$, in which case \eqref{eq:6.1} follows from 
\eqref{eq:6.2}, which we will now establish.

(b) The cases with $q,y < 3$ are trivial.
If $q < 3 \le y \le (\log x)^A$, we have
\begin{align}
 \label{eq:6.4}
  \Phi(x,y;q,a) 
 = \Phi(x,y)
  = \textstyle x\Br{\prod_{p \le y}\br{1 - \frac{1}{p}}
   + \Oh[\e^{-c\sqrt{\log x}}][A]}.
\end{align}
This follows from the analysis below, but is more straightforward 
because $s = 1$ is a simple pole, rather than a branch point, of 
the function $F(s)$ (defined in \eqref{eq:6.5}).
It also follows from stronger results of de Bruijn 
\cite{MR0035785} on $y$-pliable numbers.%
\footnote{%
Natural numbers all of whose prime factors are no greater than $y$ 
are often called $y$-{\em friable} numbers.
We propose that natural numbers all of whose prime factors 
{\em are} greater than $y$ be called $y$-{\em pliable} numbers.
}

Thus, we consider $3 \le q, y \le (\log x)^A$.
Of course, without loss of generality, we assume that 
$1 \le a \le q-1$.
For $\sigma > 1$, we define
\begin{align*}
F(s) 
 = F(s,y;q,a)
 & \defeq \sum_{n \in S(y) \cap S(q,a)} n^{-s} \\
 & \phantom{:}=
   \prod_{p \equiv a \bmod q}\br{1 - \frac{1}{p^s}}^{-1}
    \prods[p \le y][p \equiv a \bmod q]\br{1 - \frac{1}{p^s}}.
\end{align*}
The last product is clearly analytic and non-zero for $\sigma > 0$,
and using orthogonality relations for Dirichlet characters 
\eqref{eq:4.1}, \eqref{eq:5.1}, it is straightforward to verify 
that for $\sigma > 1$, 
\[
 \prod_{p \equiv a \bmod q}\br{1 - \frac{1}{p^s}}^{-1}
 = \zeta(s)^{1/\phi(q)}G(s;q,a),
\]
where
\begin{align*}
 G(s;q,a)
  & \defeq
     \zeta(s)^{-1/\phi(q)} \textstyle
      \prod_{\chi}  L(s,\chi)^{\bar{\chi}(a)/\phi(q)} \\
  & \hspace{60pt}\times
  \textstyle 
    \exp
     \Br{%
        \sum_{p \equiv a \bmod q}\sum_{\nu \ge 2}
         \frac{1}{\nu p^{\nu s}}
        -
        \underset{p^{\nu} \equiv a \bmod q}
               {\sum_{p}\sum_{\nu \ge 2}}
                 \frac{1}{\nu p^{\nu s}}
        } \\
  &\phantom{:}=
   {\textstyle
   \prod_{p \mid q}\br{1 - \frac{1}{p^s}}^{1/\phi(q)}
   \prod_{\chi \ne \chi_0} L(s,\chi)^{\bar{\chi}(a)/\phi(q)}
   } \\
  & \hspace{60pt}\times
  \textstyle 
    \exp
     \Br{%
        \sum_{p \equiv a \bmod q}\sum_{\nu \ge 2}
         \frac{1}{\nu p^{\nu s}}
        -
        \underset{p^{\nu} \equiv a \bmod q}
               {\sum_{p}\sum_{\nu \ge 2}}
                 \frac{1}{\nu p^{\nu s}}
        }.
\end{align*}
Note that $\exp\Br{\cdots} \asymp 1 + \delta^{-1}$ for 
$\sigma \ge \frac{1}{2} + \delta$ and any $\delta > 0$.
Thus, $F(s)$ admits an analytic continuation to any region for 
which $\sigma > \frac{1}{2}$ and 
$\prod_{\chi} L(s,\chi)$ is zero-free, excluding the branch point 
$s = 1$.
Namely, 
\begin{align}
 \label{eq:6.5}
 F(s) = (s-1)^{-1/\phi(q)}((s-1)\zeta(s))^{1/\phi(q)} G(s;q,a)
         \prods[p \le y][p \equiv a \bmod a]\br{1 - \frac{1}{p^s}}.
\end{align}

We may and do suppose that $A > 1$. 
Fix $\epsilon \in (0,\frac{1}{2})$, arbitrarily small, and let 
$\epsilon^{*} = \frac{\epsilon}{2A}$, so that 
$q^{\epsilon} \le (\log x)^{\epsilon/2} < (\log x)^{1/4}$.
Let $\kappa = \kappa(\epsilon^{*})$ be a positive constant, 
depending only on $\epsilon^{*}$, such that 
\[
 \mathscr{R}_{4\kappa} 
  \defeq 
   \Br{s = \sigma + i\tau : 
     \sigma \ge  
  1 - \frac{4\kappa}{q^{\epsilon^{*}}\cdot\max\{1, \log|\tau|\}}}
\]
is a zero-free region for $\prod_{\chi} L(s,\chi)$.
(By Siegel's theorem, such a constant, $\kappa$, exists but is 
ineffective.)
Define $\mathscr{R}_{2\kappa}$ and $\mathscr{R}_{\kappa}$ 
similarly.

We claim that if $\kappa$ is sufficiently small (as we will 
assume), we have the following bound and estimate.
For $s \in \mathscr{R}_{2\kappa}$ with $s \ne 1$ and 
$\sigma \ge 1 - \frac{1}{2\log y}$, we have
\begin{align}
 \label{eq:6.6}
 \begin{split}
 F(s) 
  & \ll_{\epsilon^{*}} 
   \max\{|s-1|^{-1/\phi(q)}, \br{\log(|\tau|+3)}^{1/\phi(q)}\} \\
  & \hspace{60pt} \times
     \br{\log(q|\tau|+3)}^{1 - 1/\phi(q)}
      \br{\log y}^{2/\phi(q)}.
 \end{split}
\end{align}
For $s \in \mathscr{R}_{\kappa}$ with $s \ne 1$ and 
$|s-1| \le \min\{\frac{1}{4\log y},\kappa q^{-\epsilon^{*}}\}$, we 
have
\begin{align}
 \label{eq:6.7}
  \begin{split}
   (s-1)^{1/\phi(q)}F(s)
 & = G(q,a)
   \cdot \textstyle
    \prods[p \le y][p \equiv a \bmod q]\br{1 - \frac{1}{p}} \\
 & \hspace{60pt} 
     +  \Oh[(s-1) q^{2\epsilon^{*}}  
                   (\log y)^{1+2/\phi(q)}][\epsilon^{*}].
  \end{split}
\end{align}
The bound \eqref{eq:6.6} follows from a combination of elementary 
bounds, with standard results and bounds for $\zeta(s)$ and 
$L(s,\chi)$ in the critical strip 
(cf.~\cite[Th\'eor\`eme 8.7]{MR0417077} 
or \cite[Theorem 11.4]{MR2378655}).
In particular, if $\chi$ is an exceptional character --- so that it 
is real, non-principal, and the only such character mod $q$ --- 
and if $\beta$ is its exceptional $0$ --- so that $\beta$ is real 
and satisfies $\beta \le 1 - 4\kappa q^{-\epsilon^{*}}$ --- then 
for $s \in \mathscr{R}_{2\kappa}$ with $s \ne \beta$, we have
$|L(s,\chi)|^{\pm 1} \ll \log q(|\tau| + 3) + |s - \beta|^{-1}$.
Also, by \eqref{eq:3.10} and \eqref{eq:5.6}, the inequality 
$\sums[p \le y][p \equiv a \bmod q] \frac{1}{p}
  \le \phi(q)^{-1} \log\log y + \Oh[1]$ holds, while 
$y^{\sigma} < 2$ for $\sigma \le \frac{1}{2\log y}$.
It follows that 
$\prods[p \le y][p \equiv a \bmod q](1 - p^{-s})
  \ll (\log y)^{1+2/\phi(q)}$ 
in this region.
The estimate \eqref{eq:6.7} follows similarly, by analyticity and 
Cauchy's estimate for derivatives.

Let
\[
 T 
 \defeq 
  {\textstyle%
  \exp\br{\sqrt{ \log x  }},
  }% 
 \quad
 \eta 
  \defeq 
   \frac{\kappa}{q^{\epsilon^{*}}\log T } 
  = \frac{\kappa}{q^{\epsilon^{*}}\sqrt{\log x}}.
\]
We have 
$\eta \le \min\{\frac{1}{4\log y},\kappa q^{-\epsilon^{*}}\}$ for 
$x$ sufficiently large in terms of $A$ and $\epsilon^{*}$,
so that \eqref{eq:6.6} and \eqref{eq:6.7} are applicable in the 
region with $\sigma \ge 1 - \eta$ and $|\tau| \le T$.
Let $\mathscr{C}$ be the rectangle with vertices at
$1 - \eta \pm iT, 1 + (\log x)^{-1} \pm iT$, traversed 
{\em clockwise}, but with the point $1 - \eta$ replaced by the 
truncated Hankel contour,
\[
\mathscr{H} 
  \defeq 
   [1 + \eta\e^{-\pi i}, 1 + r\e^{-\pi i}]
   \cup \{1 + r\e^{i\theta} : -\pi < \theta < \pi\}
   \cup [1 + r\e^{\pi i}, 1 + \eta \e^{\pi i}].
\]
Here, $r$ can be any number satisfying $0 < r < (\log x)^{-1}$. 
We choose $r = (\log x)^{-2}$ for now, but the integrals involved 
will be independent of $r$, and we will be able to let $r$ tend to 
$0$.
Also, let
\begin{align*}
 \mathscr{L} 
  & \defeq \{1 + (\log x)^{-1} + i\tau : -T \le \tau \le T\}.
\end{align*}

By an effective version of Perron's formula 
\cite[II \S2, Theorem 2]{MR1342300}, and our choice for $T$,
\[
  \Phi(x,y;q,a) 
 = \frac{1}{2\pi i}\int_{\mathscr{L}} F(s)x^{s}s^{-1} \dd{s}
  + \Oh[x\e^{-c\sqrt{\log x}}][\epsilon^{*},A].
\]
Using \eqref{eq:6.6} and our choice of parameters, it is 
straightforward to verify that
\[
  \int_{\mathscr{C}\setminus\mathscr{L}} F(s)x^{s}s^{-1} \dd{s}
   \ll_{\epsilon^{*},A} x\e^{-c\sqrt{\log x}}.
\]
Combining this estimate and bound with Cauchy's integral theorem, 
we see that
\begin{align}
 \label{eq:6.8}
  \Phi(x,y;q,a)
   = \frac{1}{2\pi i} \int_{\mathscr{H}} F(s)x^{s}s^{-1} \dd{s} 
    + \Oh[x\e^{-c\sqrt{\log x}}][\epsilon^{*}, A].
\end{align}
 
Next, using \eqref{eq:6.7} we obtain 
\begin{align*}
 \int_{\mathscr{H}} F(s)x^{s}s^{-1} \dd{s}  
  & =
  G(q,a) \cdot 
    \textstyle
     \prods[p \le y][p \equiv a \bmod q]\br{1 - \frac{1}{p}} \\
  & \hspace{15pt} \times
        \int_{\mathscr{H}} x^{s}(s-1)^{-1/\phi(q)} \dd{s} \\
  &\hspace{30pt}  
  + \Oh[q^{2\epsilon^{*}} (\log y)^{1 + 2/\phi(q)}
     \int_{\mathscr{H}} 
        |x^{s}(s-1)^{1 - 1/\phi(q)}| ~|\dd{s}|][\epsilon^{*},A].
\end{align*}
A change of variables, $w = (s-1)\log x$, yields
\begin{align*}
 \int_{\mathscr{H}} x^{s}(s-1)^{-1/\phi(q)} \dd{s}
  = x(\log x)^{1/\phi(q)-1} 
   \int_{\mathscr{H}^{*}} \e^{w}w^{-1/\phi(q)} \dd{w},
\end{align*}
where 
\[
\mathscr{H}^{*} 
  \defeq 
   [-\eta^{*} \e^{-\pi i}, -r^{*}\e^{-\pi i}]
   \cup \{-r^{*}\e^{i\theta} : -\pi < \theta < \pi\}
   \cup [-r^{*}\e^{\pi i}, -\eta^{*} \e^{\pi i}],
\]
with 
$\eta^{*} = -\eta\log x = -\kappa q^{-\epsilon^{*}}(\log x)^{1/2}$ 
and $r^{*} = -r\log x = -(\log x)^{-1}$.
By a standard estimate 
(cf.~\cite[II \S5, Corollary 2.1]{MR1342300})
\begin{align*}
 \frac{1}{2\pi i}
  \int_{\mathscr{H}^{*}}\e^{w}w^{-1/\phi(q)} \dd{w}
 & = \frac{1}{\Gamma(1/\phi(q))} + \Oh[\e^{-\eta^{*}/2}] %\\
%  & = \frac{1}{\Gamma(1/\phi(q))} 
%   + \textstyle\Oh[\exp\br{-\kappa q^{-\epsilon^{*}}\sqrt{\log x}}].
\end{align*}
Similarly, 
\[
 \int_{\mathscr{H}} 
        |x^{s}(s-1)^{1 - 1/\phi(q)}| ~|\dd{s}| 
  \ll_{\epsilon^{*},A} x(\log x)^{1/\phi(q) - 2}.
\]

Combining all of this with \eqref{eq:6.8} and recalling our 
choice of parameters, we obtain
\begin{align*}
 \Phi(x,y;q,a)
  & = x(\log x)^{1/\phi(q) - 1}  \\
  & \hspace{30pt} \times\textstyle
     \Br{\frac{G(q,a)}{\Gamma(1/\phi(q))}
      \prods[p \le y][p \equiv a \bmod q]\br{1 - \frac{1}{p}}
     + \Oh[\frac{ q^{2\epsilon^{*}} (\log y)^{1 + 2/\phi(q)} }
                {\log x}][\epsilon^{*},A]}.
\end{align*}
Since $\epsilon^{*} = \epsilon/2A$, where $A > 1$ and 
$\epsilon \in (0,\frac{1}{2})$ were arbitrarily chosen, the result 
follows.

(c) For $q < 3$, use \eqref{eq:6.4}. 
For $3 \le q \le (\log\log x)^A$, use \eqref{eq:6.2}.
In factoring out the term 
$\frac{G(q,a)}{\Gamma(1/\phi(q))}
  \prods[p \le y][p \equiv a \bmod q] \br{1 - p^{-1}}$, 
note that $G(q,a)^{-1} \asymp 1$ by \eqref{eq:5.7}, 
that $\frac{1}{\Gamma(1/\phi(q))} \asymp \phi(q)$, and that
$\prods[p \le y][p \equiv a \bmod q] \br{1 - p^{-1}}^{-1}
  \ll (\log y)^{1/\phi(q)}$
by \eqref{eq:3.11}. 
Thus, we consider an $\Oh$-term of order 
$q^{\epsilon}\phi(q)(\log y)^{1 + 3/\phi(q)}(\log x)^{-1}$.
\end{proof}

\part*{%
\hspace*{\fill}%
II. SPLITTING PRIMES%
\hspace*{\fill}%
}
\label{part:II}

\section{Background, notation and conventions}
 \label{sec:7}

We assume the reader has some knowledge of algebraic and analytic 
number theory, and the representation theory of finite groups.
A nominal understanding of the concepts involved in a proof of the 
Chebotarev density theorem should suffice.
For the representation theory aspect the reader may consult 
\cite{MR0450380}.
For the theory of Artin $L$-functions, we refer the reader to 
\cite{MR0218327}.

Notation and conventions remain as in \S \ref{sec:2}, unless they 
come into conflict with the following.
One caveat is that $\sigma$ and $\tau$ denote automorphisms in 
this section, though they will also continue to denote the real and 
imaginary parts of a complex number $s$ from \S\ref{sec:8}.

The following set-up is virtually copied from 
\cite{MR0447191,MR528062}.
Let $L/K$ be a Galois extension of number fields, with 
$G = \Gal{L}{K}$ denoting its Galois group.
Let $\dgr{L} \defeq [L : \Q]$ be the degree of $L$ over $\Q$, and  
let $\disc{L}{K}$ be the relative discriminant of $L$ over $K$. 
We let $\p$ denote a prime ideal of $K$ and $\pp$ denote a prime 
ideal of $L$ (lying above $\p$).
If $\p$ is unramified in $L$, i.e.~if $\p \nmid \disc{L}{K}$, we 
let the Artin symbol $\artin{L}{K}{\p}$ denote the conjugacy class 
of Frobenius automorphisms corresponding to the prime ideals $\pp$ 
lying over $\p$.

For each conjugacy class $C$ of $G$, let $\mathscr{P}(L/K,C)$ be 
the set of prime ideals $\p$ in $K$ that are unramified in $L$ and 
for which $\artin{L}{K}{\p} = C$, and let
\[
 \pi(t;L/K,C) 
  \defeq \#\{\p \in \mathscr{P}(L/K,C) : \norm{K}{\Q}{\p} \le t\}.
\]
where $\norm{K}{\Q}{\p}$ stands for the norm of $\p$ with respect 
to $K/\Q$.
By the Chebotarev density theorem, 
$\pi(t;L/K,C) \sim \frac{|C|}{|G|}\li{t}$ as $t \to \infty$.
We set
\[
  E(t;L/K,C) \defeq \pi(t;L/K,C) - \frac{|C|}{|G|}\li{t}.
\]

If $\Psi$ is any character on $G$, its associated Artin 
$L$-function $L(s,\Psi;L/K)$ is defined for $\Re{s} > 1$ by 
\[
 \log L(s,\Psi;L/K) 
  \defeq \sum_{\p} \sum_{\nu \ge 1}
          \nu^{-1} \Psi(\p^{\nu}) (\norm{K}{\Q}{\p})^{-\nu s},
\]
where the outer sum runs over all prime ideals $\p$ of $K$ and 
\[
 \Psi(\p^{\nu}) 
 = |I_{\p}|^{-1} 
    \sum_{\tau \in I_{\p}} \Psi(\Frob{\pp}^{\nu} \tau),
\]
where $I_{\p}$ is the inertia group of $\p$, and $\Frob{\pp}$ is 
any element of $\artin{L}{K}{\p}$ (it doesn't matter which one as 
characters are invariant under conjugation [they are {\em class 
functions}]).
Of course, $I_{\p}$ is trivial for primes $\p$ that are unramified 
in $L$, and we may write $L(s,\Psi;L/K)$ as a product of an 
{\em unramified}, or {\em incomplete} Artin $L$-function
$\Lun(s,\Psi;L/K)$ and a {\em ramified} Artin $L$-function
$\Lram(s,\Psi;L/K)$: for $\Re{s} > 1$, 
\begin{align*}
 \log \Lun(s,\Psi;L/K) 
  & \defeq \sum_{\p \nmid \disc{L}{K}} 
            \hspace{5pt} \sum_{\nu \ge 1}
             \nu^{-1} \Psi(\p^{\nu}) (\norm{K}{\Q}{\p})^{-\nu s} \\
  &\phantom{:}=
    \sum_{\p \nmid \disc{L}{K}} 
     \hspace{5pt} \sum_{\nu \ge 1}
      \nu^{-1}\Psi(\Frob{\pp}^{\nu})(\norm{K}{\Q}{\p})^{-\nu s} ,\\
 \log \Lram(s,\Psi;L/K) 
  & \defeq \sum_{\p \mid \disc{L}{K}} 
            \hspace{5pt} \sum_{\nu \ge 1}
             \nu^{-1} \Psi(\p^{\nu}) (\norm{K}{\Q}{\p})^{-\nu s}.
\end{align*}
These sums converge absolutely for $\Re{s} > 1$, uniformly for 
$\Re{s} \ge 1 + \delta > 1$, therefore $L(s,\Psi;L/K)$, 
$\Lun(s,\Psi;L/K)$ and $\Lram(s,\Psi;L/K)$ are analytic and 
non-zero for $\Re{s} > 1$.
They may be extended meromorphically to the entire plane.
A famous conjecture of Artin asserts that for irreducible 
characters $\Psi$, $L(s,\Psi;L/K)$ is entire, with the exception of 
a pole at $s = 1$ if $\Psi$ is trivial.
Unconditionally, it is known that $L(s,\Psi;L/K)$ is non-zero 
throughout the region $\Re{s} \ge 1$, and is also analytic 
throughout this region, except for a pole at $s = 1$ when $\Psi$ is 
trivial.
This will be sufficient for our purposes, or more precisely, for 
the purpose of proving Proposition \ref{prp:8.1}.
Letting $\chi_0$ denote the identity character of $G$, we have
$L(s,\chi_0;L/K) = \zeta_K(s)$, the Dedekind zeta-function of $K$.
Therefore
\begin{align}
 \label{eq:7.1}
 \Lun(s,\chi_0;L/K) 
 = \zeta_K(s) \textstyle
    \prod_{\p \mid \disc{L}{K}}
          \br{1 - (\norm{K}{\Q}{\p})^{-s}},
\end{align}
for $\Re{s} > 1$, and so, by analytic continuation, for all $s$.

Let $\chi$ denote an irreducible character on $G$. 
Let us abuse notation and write $\chi(C) \defeq \chi(\tau)$, where 
$\tau$ is any automorphism in $C$.
This is well-defined since characters are class functions.
Consider the product
$\prod_{\chi} \Lun(s,\chi;L/K)^{\bar{\chi}(C)}$, $\bar{\chi}$ 
denoting the complex conjugate of $\chi$.
Here and throughout, such products and analogous sums are over all 
irreducible characters on $G$.
By orthogonality relations for characters, we have, for 
$\Re{s} > 1$,
\begin{align*}
 \textstyle \log\prod_{\chi} \Lun(s,\chi;L/K)^{\bar{\chi}(C)}
 & = \sum_{\p \nmid \disc{L}{K}} \hspace{5pt} \sum_{\nu \ge 1}
      \nu^{-1} (\norm{K}{\Q}{\p})^{- \nu s} 
       \sum_{\chi} \bar{\chi}(C)\chi(\Frob{\pp}^{\nu}) \\
 & = \frac{|G|}{|C|}
      \underset{\artin{L}{K}{\p}^{\nu} =~ C}
         {\sum_{\p \nmid \disc{L}{K}} \hspace{5pt} \sum_{\nu \ge 1}}
       \nu^{-1} (\norm{K}{\Q}{\p})^{- \nu s} \\
 & = \frac{|G|}{|C|}
      \sum_{\p \in \mathscr{P}(L/K,C)}(\norm{K}{\Q}{\p})^{- s} \\
 & \hspace{60pt} + \frac{|G|}{|C|}
    \underset{\artin{L}{K}{\p}^{\nu} =~ C}
         {\sum_{\p \nmid \disc{L}{K}} \hspace{5pt} \sum_{\nu \ge 2}}
       \nu^{-1} (\norm{K}{\Q}{\p})^{- \nu s}.
\end{align*}
It follows, by \eqref{eq:7.1}, that for $\Re{s} > 1$, 
\begin{align}
 \label{eq:7.2}
 \begin{split}
 & \sum_{\p \in \mathscr{P}(L/K,C)}(\norm{K}{\Q}{\p})^{- s}
 - \frac{|C|}{|G|} \log \zeta_K(s) \\
 & \hspace{60pt} = 
    \log \mathcal{L}(s;L/K,C)
   - \underset{\artin{L}{K}{\p}^{\nu} =~ C}
        {\sum_{\p \nmid \disc{L}{K}} \hspace{5pt} \sum_{\nu \ge 2}}
         \nu^{-1} (\norm{K}{\Q}{\p})^{- \nu s},
 \end{split}
\end{align}
where for $\Re{s} > 1$, 
\begin{align}
 \label{eq:7.3}
  \begin{split}
 \mathcal{L}(s;L/K,C)
  & \defeq \textstyle
  \zeta_K(s)^{-|C|/|G|}
   \prod_{\chi} 
    \Lun(s,\chi;L/K)^{\bar{\chi}(C) \frac{|C|}{|G|}  } \\
  & \phantom{:}= \textstyle
   \prod_{\p \mid \disc{L}{K}} 
    \br{1 - (\norm{K}{\Q}{\p})^{-s}}^{ \frac{|C|}{|G|} } 
    \prod_{\chi \ne \chi_0} 
     \Lun(s,\chi;L/K)^{\bar{\chi}(C) \frac{|C|}{|G|}  }.
  \end{split}
\end{align}
It is known that for $\chi \ne \chi_0$, $L(s,\chi;L/K)$ is analytic 
and non-zero for $\Re{s} \ge 1$.
Therefore so is $\mathcal{L}(s;L/K,C)$, and
$\mathcal{L}(s;L/K,C) 
  \to \mathcal{L}(1;L/K,C) 
   \eqdef \mathcal{L}(L/K,C)$
as $s \to 1$.

\section{The main result and its corollaries}
 \label{sec:8}

Let $L/K$ be a finite Galois extension of number fields, with 
Galois group $G = \Gal{L}{K}$, and let $C$ be a conjugacy class in 
$G$.
In the notation of \S \ref{sec:7}, let $\mathcal{L}(L/K,C)$ be the 
real, positive number satisfying 
\begin{align}
 \label{eq:8.1}
 \mathcal{L}(L/K,C)^{|G| / |C|} 
  \defeq \textstyle
  \prod_{\p \mid \disc{L}{K}} 
    \br{ 1 - \frac{1}{\norm{K}{\Q}{\p}} } 
    \prod_{\chi \ne \chi_0} 
     \Lun(1,\chi;L/K)^{ \bar{\chi}(C) },
\end{align}
and let 
\begin{align}
 \label{eq:8.2}
 g(L/K,C) 
  \defeq 
   \frac{|C|}{|G|}\gamma
  + \log \mathcal{L}(L/K,C) 
  -  \underset{\artin{L}{K}{\p}^{\nu} =~ C}
        {\sum_{\p \nmid \disc{L}{K}} \hspace{5pt} \sum_{\nu \ge 2}}
         \nu^{-1} (\norm{K}{\Q}{\p})^{- \nu}.
\end{align}
Recall that $\norm{K}{\Q}{\p}$ denotes the norm of the prime ideal 
$\p$ of $K$ with respect to $\Q$.
Also recall that
$E(t;L/K,C) \defeq \pi(t;L/K,C) - \frac{|C|}{|G|}\li{t}$
is the error term in the Chebotarev density theorem.

\begin{proposition}
 \label{prp:8.1}
Let $L/K$ be a Galois extension of number fields, with Galois group 
$G = \Gal{L}{K}$, and let $C$ be a conjugacy class in $G$. 
Let $\mathscr{P}(L/K,C)$ be the set of prime ideals of $K$ that are 
unramified in $L$ and have Frobenius lying in $C$.
For $x \ge 2$, we have
\begin{align}
 \label{eq:8.3}
  \int_2^{\infty} t^{-2} E(t;L/K,C) \dd{t}
 = g(L/K,C) + \frac{|C|}{|G|}\log\log 2,
\end{align}
and
\begin{align}
 \label{eq:8.4}
 \begin{split}
  \sums[\norm{K}{\Q}{\p} \le x][\p \in \mathscr{P}(L/K,C)] 
   \frac{1}{\norm{K}{\Q}{\p}} 
 & = \frac{|C|}{|G|} \log\log x + g(L/K,C) \\
 & \hspace{30pt}
   + x^{-1} E(x;L/K,C) - \int_{x}^{\infty} t^{-2}E(t;L/K,C) \dd{t},
 \end{split}
\end{align}
with $g(L/K,C)$ as in \eqref{eq:8.2}.
\end{proposition}

We have tacitly used the fact that $\Lun(1,\chi;L/K) \ne 0$ for 
each irreducible character $\chi \ne \chi_0$, which is tantamount 
to the fact that $\mathscr{P}(L/K,C)$ contains infinitely many 
primes.
In the proof of Proposition \ref{prp:8.1} we will use nothing 
stronger than the bound 
$
 \pi(t;L/K,C) \le \#\{\p : \norm{K}{\Q}{\p} \le t\} 
               \ll t(\log t)^{-1}.
$
Using this the reader may deduce the next proposition from the one
above.
(The reader may also verify Proposition \ref{prp:8.2} directly, 
similarly to the way in which we prove Proposition \ref{prp:8.1}.)

Let 
\begin{align}
 \label{eq:8.5}
 \begin{split}
 \G(L/K,C) 
  & \defeq 
     \exp\Br{ \textstyle -\frac{|C|}{|G|}\gamma + g(L/K,C) 
         + \sum_{ \p \in \mathscr{P}(L/K,C) }
            \sum_{\nu \ge 2} \nu^{-1} (\norm{K}{\Q}{\p})^{-\nu}} \\
  & \phantom{:}= 
     \mathcal{L}(L/K,C)\cdot 
      \exp\Bigg\{\sum_{\p \in \mathscr{P}(L/K,C)}
               \sum_{\nu \ge 2} \nu^{-1} (\norm{K}{\Q}{\p})^{-\nu} \\
  & \hspace{180pt} - 
     \underset{\artin{L}{K}{\p}^{\nu} = ~ C}
        {\sum_{\p \nmid \disc{L}{K}} \hspace{5pt} \sum_{\nu \ge 2}}
         \nu^{-1} (\norm{K}{\Q}{\p})^{- \nu} \Bigg\}.
 \end{split}
\end{align}

\begin{proposition}
 \label{prp:8.2}
Let $L/K$ be a Galois extension of number fields, with Galois group 
$G = \Gal{L}{K}$, and let $C$ be a conjugacy class in $G$. 
Let $\mathscr{P}(L/K,C)$ be the set of prime ideals of $K$ that are 
unramified in $L$ and have Frobenius lying in $C$.
For $x \ge 2$, we have
\begin{align}
 \label{eq:8.6}
 \begin{split}
 & \sums[\norm{K}{\Q}{\p} \le x][\p \in \mathscr{P}(L/K,C)] 
    \log\br{1 - \frac{1}{\norm{K}{\Q}{\p}}}^{-1} \\  
 & \hspace{30pt} =
    \frac{|C|}{|G|}(\gamma + \log\log x) + \log \G(L/K,C) \\
 & \hspace{60pt}
  + x^{-1}E(x;L/K,C) 
   - \int_{x}^{\infty} t^{-2}E(t;L/K,C)  \dd{t}  
    + \Oh[(x\log x)^{-1}],
 \end{split}
\end{align}
with $\G(L/K,C)$ as in \eqref{eq:8.5}.
The implicit constant is absolute.
\end{proposition}

We have the following immediate corollary to the above 
propositions. 

\begin{theorem}
 \label{thm:8.3}
Let $L/K$ be a Galois extension of number fields, with Galois group 
$G = \Gal{L}{K}$, and let $C$ be a conjugacy class in $G$. 
Let $\mathscr{P}(L/K,C)$ be the set of prime ideals of $K$ that are 
unramified in $L$ and have Frobenius lying in $C$.
We have
\begin{align}
 \label{eq:8.7}
 \sums[\norm{K}{\Q}{\p} \le x][\p \in \mathscr{P}(L/K,C)] 
  \frac{1}{\norm{K}{\Q}{\p}} 
  =
   \frac{|C|}{|G|}\log\log x + g(L/K,C) + \oh[1][x \to \infty],
\end{align}
and
\begin{align}
 \label{eq:8.8}
 \prods[\norm{K}{\Q}{\p} \le x][\p \in \mathscr{P}(L/K,C)]
  \br{ 1 - \frac{1}{\norm{K}{\Q}{\p}} }^{-1}
  =
   \G(L/K,C)\cdot \br{\e^{\gamma} \log x}^{|C|/|G|}
             \cdot\Br{1 + \oh[1][x \to \infty]}.
\end{align}
Here, $g(L/K,C)$ and $\G(L/K,C)$ are as in \eqref{eq:8.2} and 
\eqref{eq:8.5} respectively.
\end{theorem}

Explicit error terms are of course much harder to obtain, but 
thanks to the work of others it is, for us, merely a matter of 
substitution.
In particular, by an effective version of the Chebotarev density 
theorem due to Lagarias and Odlyzko \cite[Theorem 1.3]{MR0447191}, 
and effective bounds due to Stark \cite[p.148]{MR0342472} 
(cf.~\cite[Theorem 1.4]{MR0447191}) on putative exceptional 
zeros of Dedekind zeta-functions, we have the following.

Recall that $\dgr{L} \defeq [L:\Q]$ is the degree of $L$ over $\Q$, 
and that $\zeta_L(s)$ is the Dedekind zeta-function of $L$.
Also, let $\disc{L}{\Q}$ be the absolute discriminant of $L$ over 
$\Q$.
\begin{theorem}
 \label{thm:8.4}
(a) If $\dgr{L} > 1$ then $\zeta_L(s)$ has at most one zero in the 
region defined by $s = \sigma + i\tau$ with 
$
 1 - (4\log |\disc{L}{\Q}|)^{-1} \le \sigma \le 1$, $
 |\tau| \le (4\log |\disc{L}{\Q}|)^{-1}.
$
The putative zero is denoted $\beta$, and is necessarily real and 
simple. 
(b) There exist absolute effectively computable constants $c_3$ and 
$c_4$ such that if 
$x \ge \exp\br{10\dgr{L}(\log|\disc{L}{\Q}|)^2}$, then 
\[
 \abs{\pi(x;L/K,C) - \frac{|C|}{|G|}\li{x}}
  \le \frac{|C|}{|G|}\li{x^{\beta}} \textstyle
     + c_3\exp\br{-c_4 \sqrt{\frac{\log x}{\dgr{L}}}},
\]
where the $\beta$ term is present only if $\beta$ exists.
(c) Let $m_L = 1$ if $L$ is normal over $\Q$, $m_L = 16$ if there 
is a tower of fields 
$\Q \subset K_0 \subset \cdots \subset K_r = L$ with each field 
normal over the preceding one, and $m_L = 4\dgr{L}!$ otherwise. 
If $\beta$ exists, then there exists an effectively computable 
constant $c_5$ such that 
\[
 \beta
 < \max\Br{1 - (m_L\log|\disc{L}{\Q}|)^{-1}, 
           1 - (c_5|\disc{L}{\Q}|^{1/\dgr{L}})^{-1}}.
\]
\end{theorem}

It isn't difficult to deduce that there exists an absolute 
effective constant $c_2$ such that if 
$\max\Br{m_L\log |\disc{L}{\Q}|,m_L|\disc{L}{\Q}|^{1/\dgr{L}}} 
  \le c_2\sqrt{\log x}$,
then 
\[
  \pi(x;L/K,C) = \frac{|C|}{|G|}\li{x}
 + \textstyle
    \Oh[x\exp\br{-c_4 \sqrt{\frac{\log x}{\dgr{L}}}}], 
\]
the implicit constant being absolute and effective.
Combining with Propositions \ref{prp:8.1} and \ref{prp:8.2} gives 
the following result.

\begin{theorem}
 \label{thm:8.5}
Let $L/K$ be a Galois extension of number fields, with Galois group 
$G = \Gal{L}{K}$, and let $C$ be a conjugacy class in $G$. 
Let $\mathscr{P}(L/K,C)$ be the set of prime ideals of $K$ that are 
unramified in $L$ and have Frobenius lying in $C$.
Let $\dgr{L} \defeq [L:\Q]$ and $\disc{L}{\Q}$ respectively denote 
the degree and absolute discriminant of $L$ over $\Q$.
Finally, let $m_L = 1$ if there is a tower of fields 
$\Q \subset K_0 \subset \cdots \subset K_r = L$ with each field 
normal over the preceding one, and $m_L = 4\dgr{L}!$ otherwise. 

There exist effectively computable positive absolute constants 
$c_0$ and $c_1$ such that if
$\max\Br{m_L\log |\disc{L}{\Q}|,m_L|\disc{L}{\Q}|^{1/\dgr{L}}} 
  \le c_0\sqrt{\log x}$,
then
\begin{align}
 \label{eq:8.9}
  \sums[ \norm{K}{\Q}{\p} \le x ][ \p \in \mathscr{P}(L/K,C) ] 
   \frac{1}{ \norm{K}{\Q}{\p} }
 =
  \frac{|C|}{|G|}\log\log x 
 + g(L/K,C) \textstyle
  + \Oh[\exp\br{-c_1 \sqrt{\frac{\log x}{\dgr{L}}}}],
\end{align}
and
\begin{align}
 \label{eq:8.10}
 \begin{split}
 \prods[\norm{K}{\Q}{\p} \le x][\p \in \mathscr{P}(L/K,C)]
  \br{1 -  \frac{1}{ \norm{K}{\Q}{\p} } }^{-1}
 & =
 \G(L/K,C)\cdot (\e^{\gamma}\log x)^{|C|/|G|} \\
 & \hspace{45pt} \times \textstyle
  \Br{1 + \Oh[\exp\br{-c_1 \sqrt{\frac{\log x}{\dgr{L}}}}]}.
 \end{split}
\end{align}
The implicit constants are absolute and effectively computable.
Here, $g(L/K,C)$ and $\G(L/K,C)$ are as in \eqref{eq:8.2} and 
\eqref{eq:8.5} respectively.
\end{theorem}

Naturally, larger zero-free regions for zeta functions would yield 
better results.
Indeed, Lagarias and Odlyzo \cite[Theorem 1.1]{MR0447191} proved 
that if the generalized Riemann hypothesis for the Dedekind 
zeta-function of $L$ (GRH) holds, then for $x \ge 2$, 
\[
  \pi(x;L/K,C) 
 = \frac{|C|}{|G|} \li{x}
  + \textstyle
     \Oh[\frac{|C|}{|G|} \sqrt{x}\log\br{|\disc{L}{\Q}|x^{\dgr{L}}} 
    + \log |\disc{L}{\Q}|], 
\]
the implicit constant being absolute and effectively computable.
Hence the following conditional result. 

\begin{theorem}
 \label{thm:8.6}
Let $L/K$ be a Galois extension of number fields, with Galois group 
$G = \Gal{L}{K}$, and let $C$ be a conjugacy class in $G$. 
Let $\mathscr{P}(L/K,C)$ be the set of prime ideals of $K$ that are 
unramified in $L$ and have Frobenius lying in $C$.
Let $\dgr{L} \defeq [L:\Q]$ and $\disc{L}{\Q}$ respectively denote 
the degree and absolute discriminant of $L$ over $\Q$.
Let $x \ge 2$.
On GRH for the Dedekind zeta function $\zeta_L(s)$ of $L$, we have
\begin{align}
 \label{eq:8.11}
 \begin{split}
  \sums[ \norm{K}{\Q}{\p} \le x ][ \p \in \mathscr{P}(L/K,C) ] 
   \frac{1}{ \norm{K}{\Q}{\p} }
 & =
  \frac{|C|}{|G|}\log\log x 
  + g(L/K,C) \\
 & \hspace{60pt} + \textstyle
    \Oh[\frac{|C|}{|G|} x^{-1/2}\log\br{|\disc{L}{\Q}|x^{\dgr{L}}} 
        + x^{-1}\log |\disc{L}{\Q}|],
 \end{split}
\end{align}
and
\begin{align}
 \label{eq:8.12}
 \begin{split}
 \sums[ \norm{K}{\Q}{\p} \le x ][ \p \in \mathscr{P}(L/K,C) ]  
   \log\br{1 - \frac{1}{ \norm{K}{\Q}{\p} }}^{-1} 
 & =
 \frac{|C|}{|G|}(\gamma + \log\log x) + \log \G(L/K,C) \\
 & \hspace{10pt} + \textstyle
    \Oh[\frac{|C|}{|G|} x^{-1/2}\log\br{|\disc{L}{\Q}|x^{\dgr{L}}} 
        + x^{-1}\log |\disc{L}{\Q}|].
 \end{split}
\end{align}
The implicit constants are absolute and effectively computable.
Here, $g(L/K,C)$ and $\G(L/K,C)$ are as in \eqref{eq:8.2} and 
\eqref{eq:8.5} respectively.
\end{theorem}

\section{Proof of the main result}
 \label{sec:9}

\begin{proof}[Proof of Proposition \ref{prp:8.1}] 
Recall the background and notation of \S\ref{sec:7}.
For $\sigma > 1$, let $\mathcal{L}(s;L/K,C)$ be as defined in 
\eqref{eq:7.3}, and let
\[
 g(s;L/K,C)
  \defeq 
   \frac{|C|}{|G|}\gamma
   + \log \mathcal{L}(s;L/K,C)
    - \underset{\artin{L}{K}{\p}^{\nu} =~ C}
        {\sum_{\p \nmid \disc{L}{K}} \hspace{5pt} \sum_{\nu \ge 2}}
         \nu^{-1} (\norm{K}{\Q}{\p})^{- \nu s}.
\]
These functions are in fact defined and analytic for 
$\sigma \ge 1$. Consequently, as $s \to 1$, 
$\mathcal{L}(s;L/K,C) \to \mathcal{L}(L/K,C)$ (as in 
\eqref{eq:8.1}) and $g(s;L/K,C) \to g(L/K,C)$.
Recall also that $\mathcal{L}(L/K,C) > 0$.

By \eqref{eq:7.2} --- essentially due to orthogonality relations 
for characters --- for $\sigma > 1$ we have
\[
  \sum_{\p \in \mathscr{P}(L/K,C)} \br{\norm{K}{\Q}{\p}}^{-s}
 = 
     \frac{|C|}{|G|}
      \br{\log \zeta_K(s) - \gamma} +  g(s;L/K,C).  
\]
On the other hand, for $\sigma > 1$, partial summation yields
\begin{align*}
 \sums[\norm{K}{\Q}{\p} \le y]
      [\p \in \mathscr{P}(L/K,C)]
       \br{\norm{K}{\Q}{\p}}^{-\sigma}
 & = y^{-\sigma}\pi(y;L/K,C) \\
 & \hspace{30pt}
  + {\textstyle\frac{|C|}{|G|}}\sigma 
     \int_2^y \frac{\li{t}}{t^{1+\sigma}} \dd{t}
  + \sigma \int_2^y \frac{E(t;L/K,C)}{t^{1+\sigma}} \dd{t},
\end{align*}
and since $\pi(t;L/K,C), |E(t;L/K,C)|, \li{t} < t$, letting $y$ 
tend to infinity yields
\[
 \sum_{\p \in \mathscr{P}(L/K,C)} \br{\norm{K}{\Q}{\p}}^{-\sigma}
 = 
  {\textstyle\frac{|C|}{|G|}}\sigma 
   \int_2^{\infty} \frac{\li{t}}{t^{1+\sigma}} \dd{t}
  + \sigma \int_2^{\infty} \frac{E(t;L/K,C)}{t^{1+\sigma}} \dd{t}.
\]
Integration by parts, followed by the substitution $u = \log t$, 
followed by an application of Lemma \ref{lem:4.1}, followed by 
\eqref{eq:4.2}, yields
\begin{align*}
    {\textstyle\frac{|C|}{|G|}}\sigma 
   \int_{2}^{\infty} \frac{\li{t}}{t^{1+\sigma}} \dd{t} 
 & = 
   {\textstyle\frac{|C|}{|G|}}
   \int_{2}^{\infty} \frac{\dd{t}}{t^{\sigma}\log t}  
  = 
    {\textstyle\frac{|C|}{|G|}}
  \int_{\log 2}^{\infty} u^{-1}\e^{-(\sigma - 1)u} \dd{u} \\
 & = 
   \textstyle
   \frac{|C|}{|G|}
   \br{%
    \log{\br{\frac{1}{\sigma - 1}}} 
   - \gamma 
    - \log\log 2 
     + \oh[1][\sigma \to 1]%
     } \\
  & = 
  \textstyle
   \frac{|C|}{|G|} 
   \br{%
    \log{\zeta(\sigma)} 
   - \gamma 
    - \log\log 2 
     + \oh[1][\sigma \to 1]%
     }.
\end{align*}
Combining all of this, we obtain
\[
 \sigma \int_2^{\infty} \frac{E(t;L/K,C)}{t^{1+\sigma}} \dd{t}
 =  
   g(\sigma;L/K,C) 
  + \frac{|C|}{|G|}
     \br{\log\log 2 
        + \oh[1][\sigma \to 1]}.
\]
Since $\pi(t;L/K,C), \li{t} \ll t(\log t)^{-1}$, we have
$t^{-2}E(t;L/K,C) \ll (t\log t)^{-1}$.
Therefore, by Lemma \ref{lem:4.2}, 
$\int_2^{\infty} t^{-2}E(t;L/K,C)\dd{t}$ converges, and 
\eqref{eq:8.3} follows.

Partial summation yields
\[
 \sums[ \norm{K}{\Q}{\p} \le x ][ \p \in \mathscr{P}(L/K,C) ] 
  \frac{1}{ \norm{K}{\Q}{\p} }
  = 
   x^{-1}\pi(x;L/K,C) 
   + {\textstyle\frac{|C|}{|G|}}
      \int_2^x t^{-2}\li{t} \dd{t} 
     + \int_2^x t^{-2}E(t;L/K,C) \dd{t}.
\]
We have
$x^{-1}\pi(x;L/K,C) 
 = \frac{|C|}{|G|} x^{-1}\li{x} + x^{-1} E(x;L/K,C)$;
integration by parts yields
\[
 {\textstyle\frac{|C|}{|G|}}
      \int_2^x t^{-2}\li{t} \dd{t}
 = 
   \frac{|C|}{|G|} 
    \br{-x^{-1}\li{x} + \log\log x - \log\log 2};
\]
and by \eqref{eq:8.3},
\[
 \int_2^x t^{-2}E(t;L/K,C) \dd{t}
 = 
  g(L/K,C) 
 + \frac{|C|}{|G|}\log\log 2
 - \int_x^{\infty} t^{-2}E(t;L/K,C) \dd{t}.
\]
Combining yields \eqref{eq:8.4}.
\end{proof}

\section{More about $g(L/K,C)$ and $\G(L/K,C)$}
 \label{sec:10}

Suppose that $K = \Q$ and that $C = \{1\}$ is trivial.
Instead abusing notation we can revert to writing $\bar{\chi}(1)$ 
instead of $\bar{\chi}(C)$.
Let us write $\mathcal{L}$ as short-hand for 
$\mathcal{L}(L/\Q;\{1\})$, $L(s,\chi)$ for $L(s,\chi;L/\Q)$, 
$\Delta_L$ for $\disc{L}{\Q}$ (the absolute discriminant of $L$ 
over $\Q$), etc.
Recall that $\dgr{L} = [L : \Q] = |G|$.
By \eqref{eq:7.3} and analyticity, we have
\begin{align*}
 \textstyle
  \mathcal{L}^{\dgr{L}} \cdot
   \prod_{\chi \ne \chi_0} \Lram(1,\chi)^{\bar{\chi}(1)} 
 & = \lim_{s \to 1} \textstyle
      \prod_{p \mid \Delta_L} \br{1 - \frac{1}{p^s}}
       \prod_{\chi \ne \chi_0} L(s,\chi)^{\bar{\chi}(1)} \\
 & = \lim_{s \to 1} \textstyle 
      \zeta(s)^{-1}\prod_{\chi} L(s,\chi)^{\bar{\chi}(1)}.
\end{align*}
Now, it is a standard result that
$\prod_{\chi} L(s,\chi)^{\bar{\chi}(1)} = \zeta_L(s)$ 
(and generally that 
$\prod_{\chi} L(s,\chi;L/K)^{\bar{\chi}(1)} = \zeta_L(s)$.)
The Dedekind zeta-function $\zeta_L(s)$ extends to a meromorphic 
function to the entire plane, with just one pole at $s = 1$, which 
is a simple pole, with residue given by the {\em analytic class 
number formula}.
We have
\[
 \mathcal{L}^{\dgr{L}} \cdot {\textstyle
   \prod_{\chi \ne \chi_0} \Lram(1,\chi)^{\bar{\chi}(1)}}
 = \Res{s=1} \zeta_L(s)
 = \frac{2^{r_1}(2\pi)^{r_2}h_L\Reg{L}}{w_L\sqrt{|\Delta_L|}},
\]
where: $\dgr{L} = r_1 + 2r_2$, $r_1$ and $2r_2$ respectively 
denoting the number of real and complex embeddings of $L$;
$h_L$ is the {\em class number} of $L$;
$\Reg{L}$ is the {\em regulator} of $L$; and
$w_L$ is the number of roots of unity in $L$.

Recalling the definition \eqref{eq:8.2} of $g = g(L/\Q,\{1\})$, we 
have
\begin{align*}
 g 
 & = \textstyle 
      \frac{\gamma}{\dgr{L}} + 
       \frac{1}{\dgr{L}}
        \log\br{\frac{2^{r_1}(2\pi)^{r_2}h_L\Reg{L}}
                     {w_L\sqrt{|\Delta_L|}}} \\
 & \hspace{60pt} - 
    \frac{1}{\dgr{L}}
     \sum_{\chi \ne \chi_0} \bar{\chi}(1)\log \Lram(1,\chi)
    - \underset{ \sigma_{\pp}^{\nu} = 1}
         {\sum_{p \nmid \Delta_L}  \sum_{\nu \ge 2}}
           \frac{1}{\nu p^{\nu}},
\end{align*}
where $\sigma_{\pp}$ denotes the Frobenius element of $\pp$, and 
$\pp$ is any prime in $L$ lying above $p$.
Recalling the definition \eqref{eq:8.5} of $\G = \G(L/\Q,\{1\})$, 
we have
\begin{align*}
 \G
 & = \textstyle
      \br{ \frac{2^{r_1}(2\pi)^{r_2}h_L\Reg{L}}
                {w_L\sqrt{|\Delta_L|}} }^{ 1/\dgr{L} }
        \prod_{\chi \ne \chi_0} 
         \Lram(1,\chi)^{ -\bar{\chi}(1)/\dgr{L}} \\
 & \hspace{60pt} \times \textstyle
    \exp\Bigg\{
             \sum_{ p \in \mathscr{P} }
              \sum_{ \nu \ge 2 } \frac{1}{ \nu p^{\nu} }  
             - \underset{ \sigma_{\pp}^{\nu} = ~ 1 }
              { \sum_{ p \nmid \Delta_L } \sum_{ \nu \ge 2 } }
                 \frac{1}{ \nu p^{\nu} }  
         \Bigg\} \\
 & = \textstyle
      \br{ \frac{2^{r_1}(2\pi)^{r_2}h_L\Reg{L}}
                {w_L\sqrt{|\Delta_L|}} }^{ 1/\dgr{L} }
        \prod_{\chi \ne \chi_0} 
         \Lram(1,\chi)^{ -\bar{\chi}(1)/\dgr{L}} \\
 & \hspace{120pt} \times \textstyle
    \exp\Bigg\{  
             - \underset{ \sigma_{\pp}  \ne 1,~
                           \sigma_{\pp}^{\nu} = 1 }
              { \sum_{ p \nmid \Delta_L } \sum_{ \nu \ge 2 } }
                 \frac{1}{ \nu p^{\nu} }  
         \Bigg\}.
\end{align*}

We leave the interested reader to find interesting interpretations 
of $g(L/K,C)$ and $\G(L/K,C)$ in other cases.

% \bibliography{bardestani-freiberg-mertens-theorem.bib}{}
% \bibliographystyle{plain}

\end{document}